\documentclass[12pt,a4paper]{article}%
\usepackage{amsmath, amsfonts, amsthm,color, latexsym, amssymb}
\usepackage{amscd}
\usepackage{amsmath}
\usepackage{amsfonts}
\usepackage{amssymb}
\usepackage{graphicx}%
\providecommand{\U}[1]{\protect\rule{.1in}{.1in}}
\providecommand{\U}[1]{\protect\rule{.1in}{.1in}}
\providecommand{\U}[1]{\protect\rule{.1in}{.1in}}
\newtheorem{theorem}{Theorem}[section]
\newtheorem{proposition}[theorem]{Proposition}
\newtheorem{corollary}[theorem]{Corollary}
\newtheorem{example}[theorem]{Example}

\newtheorem{remark}[theorem]{Remark}

\newtheorem{lemma}[theorem]{Lemma}
\newtheorem{definition}[theorem]{Definition}
\begin{document}

\title{Spaceability in Banach and quasi-Banach sequence spaces}
\author{G. Botelho\thanks{Supported by CNPq Grant 306981/2008-4 and
INCT-Matem\'{a}tica.}\thinspace, D. Diniz, V. V. F\'{a}varo\thanks{Supported
by Fapemig Grant CEX-APQ-00208-09.}\thinspace, D. Pellegrino\thanks{Supported
by INCT-Matem\'{a}tica, PROCAD-NF-Capes, CNPq Grants 620108/2008-8 (Ed.
Casadinho) and 301237/2009-3.\hfill\newline\indent2010 Mathematics Subject
Classification: 46A45, 46A16, 46B45.}}
\date{}
\maketitle

\begin{abstract}
Let $X$ be a Banach space. We prove that, for a large class of Banach or
quasi-Banach spaces $E$ of $X$-valued sequences, the sets $E-\bigcup
_{q\in\Gamma}\ell_{q}(X)$, where $\Gamma$ is any subset of $(0,\infty]$, and
$E-c_{0}(X)$ contain closed infinite-dimensional subspaces of $E$ (if
non-empty, of course). This result is applied in several particular cases and
it is also shown that the same technique can be used to improve a result on
the existence of spaces formed by norm-attaining linear operators.

\end{abstract}


\section*{Introduction}

A subset $A$ of a Banach or quasi-Banach space $E$ is $\mu$\textit{-lineable}
(\textit{spaceable}) if $A\cup\{0\}$ contains a $\mu$-dimensional (closed
infinite-dimensional) linear subspace of $E$. The last few years have
witnessed the appearance of lots of papers concerning lineability and
spaceability (see, for example, \cite{Aron2, Aron, BG, lap, seo}). The aim of
this paper is to explore a technique to prove lineability and spaceability
that can be applied in several different settings. It is our opinion that this
technique was first used in the context of lineability/spaceability in our
preprint \cite{arxiv}, of which this paper is an improved version.

Let $c$ denote the cardinality of the set of real numbers $\mathbb{R}$. In
\cite{lap} it is proved that $\ell_{p}-\ell_{q}$ is $c$-lineable for every
$p>q\geq1$. With the help of \cite{timoney} this result can be substantially
improved in the sense that $\ell_{p}-\bigcup_{1\leq q<p}\ell_{q}$ is spaceable
for every $p>1$. In this paper we address the following questions: What about
the non-locally convex range $0<p<1$? Can these results be generalized to
sequence spaces other than $\ell_{p}$?

As to the first question, it is worth recalling that the structure of
quasi-Banach spaces (or, more generally, metrizable complete tvs, called
$F$-spaces) is quite different from the structure of Banach spaces. For our
purposes, the consequence is that the extension of lineability/spaceability
arguments from Banach to quasi-Banach spaces is not straightforward in
general. For example, in \cite[Section 6]{Wilanski} it is essentially proved
(with a different terminology) that if $Y$ is a closed infinite-codimensional
linear subspace of the Banach space $X$, then $X-Y$ is spaceable. A
counterexample due to Kalton \cite[Theorem 1.1]{Kalton} shows that this result
is not valid for quasi-Banach spaces (there exists a quasi-Banach space $K$
with an $1$-dimensional subspace that is contained in all closed
infinite-dimensional subspaces of $K$). Besides, the search for closed
infinite-dimensional subspaces of quasi-Banach spaces is a quite delicate
issue. Even fundamental facts are unknown, for example the following problem
is still open (cf. \cite[Problem 3.1]{hand}): Does every
(infinite-dimensional) quasi-Banach space have a proper closed
infinite-dimensional subspace? Nevertheless we solve the first question in the
positive: as a particular case of our results we get that $\ell_{p}%
-\bigcup_{1\leq q<p}\ell_{q}$ is spaceable for every $p>0$ (cf. Corollary
\ref{corol}).

As to the second question, we identify a large class of vector-valued sequence
spaces, called \textit{invariant sequence spaces} (cf. Definition \ref{defi}),
such that if $E$ is an invariant Banach or quasi-Banach space of $X$-valued
sequences, where $X$ is a Banach space, then the sets $E - \bigcup_{q
\in\Gamma}\ell_{q}(X)$, where $\Gamma$ is any subset of $(0, \infty]$, and
$E-c_{0}(X)$ are spaceable whenever they are non-empty (cf. Theorem
\ref{theorem}). Several classical sequence spaces are invariant sequence
spaces (cf. Example \ref{exam}).

In order to make clear that the technique we use can be useful in a variety of
other situations, we finish the paper with an application to the
$c$-lineability of sets of norm-attaining linear operators (cf. Proposition
\ref{Lineability of NA}).

From now on all Banach and quasi-Banach spaces are considered over a fixed
scalar field $\mathbb{K}$ which can be either $\mathbb{R}$ or $\mathbb{C}$.

\section{Sequence spaces}

In this section we introduce a quite general class of scalar-valued or
vector-valued sequence spaces and prove that certain of their remarkable
subsets have spaceable complements.

\begin{definition}
\label{defi}\rm Let $X\neq\{0\}$ be a Banach space.\newline(a) Given $x\in
X^{\mathbb{N}}$, by $x^{0}$ we mean the zerofree version of $x$, that is: if
$x$ has only finitely many non-zero coordinates, then $x^{0}=0$; otherwise,
$x^{0}=(x_{j})_{j=1}^{\infty}$ where $x_{j}$ is the $j$-th non-zero coordinate
of $x$.\newline(b) By an \textit{invariant sequence space over $X$} we mean an
infinite-dimensional Banach or quasi-Banach space $E$ of $X$-valued sequences
enjoying the following conditions:\newline(b1) For $x\in X^{\mathbb{N}}$ such
that $x^{0}\neq0$, $x\in E$ if and only if $x^{0}\in E$, and in this case
$\Vert x\Vert\leq K\Vert x^{0}\Vert$ for some constant $K$ depending only on
$E$.\newline(b2) $\Vert x_{j}\Vert_{X}\leq\Vert x\Vert_{E}$ for every
$x=(x_{j})_{j=1}^{\infty}\in E$ and every $j\in\mathbb{N}$. \newline An
\textit{invariant sequence space} is an invariant sequence space over some
Banach space $X$.
\end{definition}

Several classical sequence spaces are invariant sequence spaces:

\begin{example}
\label{exam}\rm(a) Given a Banach space $X$, it is obvious that for every
$0<p\leq\infty$, $\ell_{p}(X)$ (absolutely $p$-summable $X$-valued sequences),
$\ell_{p}^{u}(X)$ (unconditionally $p$-summable $X$-valued sequences) and
$\ell_{p}^{w}(X)$ (weakly $p$-summable $X$-valued sequences) are invariant
sequence spaces over X with their respective usual norms ($p$-norms if $0 < p
< 1$). In particular, $\ell_{p}$, $0<p\leq\infty$, are invariant sequence
spaces (over $\mathbb{K}$).\newline(b) The Lorentz spaces $\ell_{p,q}$,
$0<p<\infty$, $0<q<\infty$ (see, e.g., \cite[13.9.1]{livropietsch}). It is
easy to see that these spaces are invariant sequence spaces (over $\mathbb{K}%
$): indeed, given $0\neq x^{0}\in\ell_{p,q}$, the non-increasing rearrangement
of $x$ coincides with that of $x^{0}$. So $\Vert x\Vert_{p,q}=\Vert x^{0}%
\Vert_{p,q}<\infty$. \newline(c) The Orlicz sequence spaces (see, e.g.,
\cite[4.a.1]{lt}). Let $M$ be an Orlicz function and $\ell_{M}$ be the
corresponding Orlicz sequence space. The condition $M(0)=0$ makes clear that
$\ell_{M}$ is an invariant sequence space (over $\mathbb{K}$). For the same
reason, its closed subspace $h_{M}$ is an invariant sequence space as
well.\newline(d) Mixed sequence spaces (see, e.g., \cite[16.4]{livropietsch}).
Given $0<p\leq s\leq\infty$ and a Banach space $X$, by $\ell_{m(s;p)}\left(
X\right)  $ we mean the Banach ($p$-Banach if $0<p<1$) space of all mixed
$(s,p)$-summable sequences on $X$. It is not difficult to see that
$\ell_{m(s;p)}\left(  X\right)  $ is an invariant sequence space over $X$.
\end{example}

Now we can prove our main result. Given an invariant sequence space $E$ over
the Banach space $X$, regarding both $E$ and $\ell_{p}(X)$ as subsets of
$X^{\mathbb{N}}$, we can talk about the difference $E-\ell_{p}(X)$ and related ones.

\begin{theorem}
\label{theorem} Let $E$ be an invariant sequence space over the Banach space
$X$. Then\\
{\rm(a)} For every $\Gamma\subseteq(0,\infty]$,
$E-\bigcup_{q\in\Gamma}\ell_{q}(X)$ is either empty or spaceable.\\
{\rm(b)} $E - c_{0}(X)$ is either empty or spaceable.
\end{theorem}

\begin{proof}
Put $A=\bigcup_{q\in\Gamma}\ell_{q}(X)$ in (a) and $A = c_0(X)$ in (b). Assume that $E-A$ is non-empty and
choose $x\in E-A$. Since $E$ is an invariant sequence space, $x^{0}\in E$, and
obviously $x^{0}\notin A$. Writing $x^{0}=(x_{j})_{j=1}^{\infty}$ we have that
$x^{0}\in E-A$ and $x_{j}\neq0$ for every $j$. Split $\mathbb{N}$ into
countably many infinite pairwise disjoint subsets $(\mathbb{N}_{i}%
)_{i=1}^{\infty}$. For every $i\in\mathbb{N}$ set $\mathbb{N}_{i}%
=\{i_{1}<i_{2}<\ldots\}$ and define
\[
y_{i}=\sum_{j=1}^{\infty}x_{j}e_{i_{j}}\in X^{\mathbb{N}}.
\]
Observe that $y_{i}^{0}=x^{0}$ for every $i$. So $0\neq y_{i}^{0}\in E$ for
every $i$. Hence each $y_{i}\in E$ because $E$ is an invariant sequence space.
Let us see that $y_{i}\notin A$: in (a) this occurs because $\Vert y_{i}\Vert_{r}=\Vert x^{0}\Vert_{r}=\Vert
x\Vert_{r}$ for every $0 < r \leq \infty$ and in (b) because $\|x_j\| \nrightarrow 0$. Let $K$ be the constant of condition
\ref{defi}(b1) and define $\tilde{s}=1$ if $E$ is a Banach space and
$\tilde{s}=s$ if $E$ is a $s$-Banach space, $0<s<1$. For $(a_{j}%
)_{j=1}^{\infty}\in\ell_{\tilde{s}}$,
\begin{align*}
\sum_{j=1}^{\infty}\Vert a_{j}y_{j}\Vert^{\tilde{s}}  & =\sum_{j=1}^{\infty
}|a_{j}|^{\tilde{s}}\Vert y_{j}\Vert^{\tilde{s}}\leq K^{\tilde{s}}\sum
_{j=1}^{\infty}\left\vert a_{j}\right\vert ^{\tilde{s}}\left\Vert y_{j}%
^{0}\right\Vert ^{\tilde{s}}\\
& =K^{\tilde{s}}\left\Vert x^{0}\right\Vert ^{\tilde{s}}\sum_{j=1}^{\infty
}\left\vert a_{j}\right\vert ^{\tilde{s}}=K^{\tilde{s}}\left\Vert
x^{0}\right\Vert ^{\tilde{s}}\left\Vert (a_{j})_{j=1}^{\infty}\right\Vert
_{\tilde{s}}^{\tilde{s}}<\infty.
\end{align*}
Thus $\sum_{j=1}^{\infty}\Vert a_{j}y_{j}\Vert<\infty$ if $E$ is a Banach
space and $\sum_{j=1}^{\infty}\Vert a_{j}y_{j}\Vert^{{s}}<\infty$ if $E$ is a
$s$-Banach space, $0<s<1$. In both cases the series $\sum_{j=1}^{\infty}%
a_{j}y_{j}$ converges in $E$, hence the operator
\[
T\colon\ell_{\tilde{s}}\longrightarrow E~~,~~T\left(  \left(  a_{j}\right)
_{j=1}^{\infty}\right)  =\sum\limits_{j=1}^{\infty}a_{j}y_{j}%
\]
is well defined. It is easy to see that $T$ is linear and injective. Thus
$\overline{T\left(  \ell_{\tilde{s}}\right)  }$ is a closed
infinite-dimensional subspace of $E$.
We just have to show that $\overline{T\left(  \ell_{\tilde{s}}\right)
}-\left\{  0\right\}  \subseteq E-A$. Let $z=\left(  z_{n}\right)
_{n=1}^{\infty}\in\overline{T\left(  \ell_{\tilde{s}}\right)  },$ $z\neq0$.
There are sequences $\left(  a_{i}^{(k)}\right)  _{i=1}^{\infty}\in
\ell_{\tilde{s}}$, $k\in\mathbb{N}$, such that $z=\lim_{k\rightarrow\infty
}T\left(  \left(  a_{i}^{(k)}\right)  _{i=1}^{\infty}\right)  $ in $E.$ Note
that, for each $k\in\mathbb{N}$,%
\[
T\left(  \left(  a_{i}^{(k)}\right)  _{i=1}^{\infty}\right)  =\sum
\limits_{i=1}^{\infty}a_{i}^{(k)}y_{i}=\sum\limits_{i=1}^{\infty}a_{i}%
^{(k)}\sum\limits_{j=1}^{\infty}x_{j}e_{i_{j}}=\sum\limits_{i=1}^{\infty}%
\sum\limits_{j=1}^{\infty}a_{i}^{(k)}x_{j}e_{i_{j}}.
\]
Fix $r\in\mathbb{N}$ such that $z_{r}\neq0.$ Since $\mathbb{N}=\bigcup
_{j=1}^{\infty}\mathbb{N}_{j}$, there are (unique) $m,t\in\mathbb{N}$ such that
$e_{m_{t}}=e_{r}$. Thus, for each $k\in\mathbb{N}$, the $r$-th coordinate of
$T\left(  \left(  a_{i}^{(k)}\right)  _{i=1}^{\infty}\right)  $ is the number
$a_{m}^{(k)}x_{t}.$ Condition \ref{defi}(b2) assures that convergence in $E$
implies coordinatewise convergence, so
\[
z_{r}=\lim_{k\rightarrow\infty}a_{m}^{(k)}x_{t}=x_{t}\cdot\lim_{k\rightarrow
\infty}a_{m}^{(k)}.
\]
It follows that $x_{t}\neq0.$ Hence $\lim_{k\rightarrow\infty}|a_{m}%
^{(k)}|=\frac{\left\Vert z_{r}\right\Vert }{\left\Vert x_{t}\right\Vert }%
\neq0$. For $j,k\in\mathbb{N}$, the $m_{j}$-th coordinate of $T\left(  \left(
a_{i}^{(k)}\right)  _{i=1}^{\infty}\right)  $ is $a_{m}^{(k)}x_{j}.$ Defining
$\alpha_{m}=\frac{\left\Vert z_{r}\right\Vert }{\left\Vert x_{t}\right\Vert
}\neq0,$
\[
\lim_{k\rightarrow\infty}\Vert a_{m}^{(k)}x_{j}\Vert=\lim_{k\rightarrow\infty
}|a_{m}^{(k)}|\Vert x_{j}\Vert=\Vert x_{j}\Vert\cdot\lim_{k\rightarrow\infty
}|a_{m}^{(k)}|=\alpha_{m}\left\Vert x_{j}\right\Vert
\]
for every $j\in\mathbb{N}$. On the other hand, coordinatewise convergence
gives $\lim_{k\rightarrow\infty}\Vert a_{m}^{(k)}x_{j}\Vert=\Vert z_{m_{j}%
}\Vert$, so $\Vert z_{m_{j}}\Vert=\alpha_{m}\Vert x_{j}\Vert$ for each
$j\in\mathbb{N}$. Observe that $m$, which depends on $r$, is fixed, so the
natural numbers $(m_{j})_{j=1}^{\infty}$ are pairwise distinct (remember that
$\mathbb{N}_{m}=\{m_{1}<m_{2}<\ldots\}$).\\
(a) As $x^{0}\notin A$, we have $\Vert
x^{0}\Vert_{q}=\infty$ for all $q\in\Gamma$. Assume first that $\infty \notin \Gamma$. In this case,
\[
\left\Vert z\right\Vert _{q}^{q}=\sum\limits_{n=1}^{\infty}\left\Vert
z_{n}\right\Vert ^{q}\geq\sum\limits_{j=1}^{\infty}\left\Vert z_{m_{j}%
}\right\Vert ^{q}=\sum\limits_{j=1}^{\infty}\alpha_{m}^{q}\cdot\left\Vert
x_{j}\right\Vert ^{q}=\alpha_{m}^{q}\cdot\left\Vert x^{0}\right\Vert _{q}%
^{q}=\infty,
\]
for all $q\in\Gamma$, proving that $z \notin \bigcup_{q\in\Gamma}\ell_{q}(X)$. If $\infty \in \Gamma$,
$$\|z\|_\infty = \sup_n \|z_n\| \geq \sup_j \|z_{m_j}\| = \alpha_m \cdot \sup_j \|x_j\| = \alpha_m \|x^0\|_\infty = \infty  ,$$
proving again that $z \notin \bigcup_{q\in\Gamma}\ell_{q}(X)$.\\
(b) As $x^{0}\notin A$, we have $\|x_j\| \nrightarrow 0$. Since $(\Vert z_{m_{j}}\|)_{j=1}^\infty $ is a subsequence of $(\Vert z_n \|)_{n=1}^\infty $ , $\Vert z_{m_{j}}\Vert=\alpha_{m}\Vert x_{j}\Vert$ for every $j$ and $\alpha_m \neq 0$, it is clear that $\|z_n\| \nrightarrow 0$. Thus $z \notin c_0(X)$.\\
\indent Therefore $z\notin A$ in both cases, so $\overline{T\left(
\ell_{\tilde{s}}\right)  }-\left\{  0\right\}  \subseteq E-A$.
\end{proof}


We list a few consequences.


When we write $F \subset E$ we mean that $E$ contains $F$ as a linear subspace
and $E \neq F$. We are not asking neither $E$ to contain an isomorphic copy of
$F$ nor the inclusion $F \hookrightarrow E$ to be continuous.

\begin{corollary}
Let $E$ be an invariant sequence space over $\mathbb{K}$.\\
{\rm(a)}
If $0 < p \leq\infty$ and $\ell_{p} \subset E$, then $E - \ell_{p}$ is
spaceable.\\
{\rm (b)} If $c_{0} \subset E$, then $E - c_{0}$ is spaceable.
\end{corollary}

From the results due to Kitson and Timoney \cite{timoney} we derive that
$\ell_{p}^{u}(X) - \ell_{p}(X)$ for $p \geq1$, and $\ell_{p}- \bigcup
\limits_{0<q<p}\ell_{q}$ for $p > 1$, are spaceable. However, as is made clear
in \cite[Remark 2.2]{timoney}, their results are restricted to Fr\'echet
spaces (see the Introduction). Next we extend the spaceability of $\ell
_{p}^{u}(X) - \ell_{p}(X)$ and $\ell_{p}-\bigcup\limits_{0<q<p}\ell_{q}$ to
the non-locally convex case:

\begin{corollary}
$\ell_{m(s;p)}\left(  X\right)  - \ell_{p}(X) $ and $\ell_{p}^{u}(X) -
\ell_{p}(X)$ are spaceable for $0 < p \leq s < \infty$ and every
infinite-dimensional Banach space $X$. Hence $\ell_{p}^{w}(X) - \ell_{p}(X)$
is spaceable as well.
\end{corollary}

\begin{proof} By \cite[Proposition 1.2(1)]{mario} we have that $\ell_{m(\infty;p)}\left(  X\right) = \ell_p(X) \subseteq \ell_{m(s;p)}\left(  X\right)  $, and by \cite[Theorem 2.1]{mario}, $\ell_{m(s;p)}\left(  X\right) \neq \ell_p(X)$. On the other hand, the identity operator on any infinite-dimensional Banach space fails to
be absolutely $p$-summing for every $0 < p < \infty$ (the case $1 \leq p <
\infty$ is well known, and the case $0 < p < 1$ follows from the fact that
$p$-summing operators are $q$-summing whenever $p \leq q$). So $\ell_{p}%
^{u}(X) \neq\ell_{p}(X)$. As $\ell_{m(s;p)}\left(  X\right)  $ and $\ell_{p}^{u}(X)$ are invariant sequence spaces
over $X$, the first assertion follows from Theorem \ref{theorem}. As $\ell_{p}^{u}(X)
\subseteq\ell_{p}^{w}(X)$, the second assertion follows.
\end{proof}

Before proving that $\ell_{p}-\bigcup\limits_{0<q<p}\ell_{q}$ is spaceable we
have to check first that it is non-empty. Although we think this is folklore,
we have not been able to find a reference in the literature. So, for the sake
of completeness, we include a short proof, which was kindly communicated to us
by M. C. Matos.

\begin{lemma}
\label{lemmanovo} $\ell_{p}-\bigcup\limits_{0<q<p}\ell_{q} \neq\emptyset$ for
every $p > 0$.
\end{lemma}

\begin{proof}
Since $\left(  \frac{1}{\sqrt{n}}\right)  _{n=1}^{\infty}\notin\ell_{2}$ and
$\left(  \frac{1}{\sqrt{n}}\right)  _{n=1}^{\infty}\in\ell_{r}$ for all $r>2$,
for each $\left(  y_{n}\right)  _{n=1}^{\infty}\in\ell_{q},$ $0<q<2,$ it
follows from H\"{o}lder's inequality that%
\[
\sum\limits_{n=1}^{\infty}\left\vert \frac{1}{\sqrt{n}}y_{n}\right\vert
\leq\left\Vert \left(  \frac{1}{\sqrt{n}}\right)  _{n=1}^{\infty}\right\Vert
_{q^{\prime}}\cdot\left\Vert \left(  y_{n}\right)  _{n=1}^{\infty}\right\Vert
_{q}<\infty.
\]
Supposing that $\ell_{2}=\bigcup\limits_{0<q<2}\ell_{q},$ we have that
$\sum\limits_{n=1}^{\infty}\left\vert \frac{1}{\sqrt{n}}y_{n}\right\vert
<\infty$ for every $\left(  y_{n}\right)  _{n=1}^{\infty}\in$ $\ell_{2}$. So,
consider, for each positive integer $k$, the continuous linear functional on
$\ell_{2}$ defined by $T_{k}\left(  \left(  y_{n}\right)  _{n=1}^{\infty
}\right)  =\sum\limits_{n=1}^{k}\frac{1}{\sqrt{n}}y_{n}.$ As
\[
\sup_{k\in\mathbb{N}}\left\vert T_{k}\left(  \left(  y_{n}\right)
_{n=1}^{\infty}\right)  \right\vert =\sum\limits_{n=1}^{\infty}\left\vert
\frac{1}{\sqrt{n}}y_{n}\right\vert <\infty
\]
for each $\left(  y_{n}\right)  _{n=1}^{\infty}\in\ell_{2}$, by the
Banach-Steinhaus Theorem we conclude that
\[
T\left(  \left(  y_{n}\right)  _{n=1}^{\infty}\right)  =\lim_{k}T_{k}\left(
\left(  y_{n}\right)  _{n=1}^{\infty}\right)  =\sum\limits_{n=1}^{\infty}%
\frac{1}{\sqrt{n}}y_{n}%
\]
defines a continuous linear functional on $\ell_{2}$ and it follows that
$\left(  \frac{1}{\sqrt{n}}\right)  _{n=1}^{\infty}\in\ell_{2}$ - a
contradiction which proves that there is $x\in\ell_{2}-\bigcup
\limits_{0<q<2}\ell_{q}.$ So $\left(  \left\vert x_{n}\right\vert ^{\frac
{2}{p}}\right)  _{n=1}^{\infty}\in\ell_{p}-\bigcup\limits_{0<q<p}%
\ell_{q}.$
\end{proof}

\begin{corollary}
\label{corol} $\ell_{p}-\bigcup\limits_{0<q<p}\ell_{q}$ is spaceable for every
$p > 0$.
\end{corollary}

\begin{proof} We know that $\ell_{p}$ is an invariant sequence space over
$\mathbb{K}$ and from Lemma \ref{lemmanovo} we have $\ell_{p}-
\bigcup\limits_{0<q<p}\ell_{q} \neq\emptyset$. The result follows from Theorem
\ref{theorem}.
\end{proof}

\begin{remark}
\rm Theorem \ref{theorem} can be applied in a variety of other situations.
For example, for Lorentz spaces it applies to $\ell_{q,r} - \ell_{p}$ for $0 <
p < q$ and $r > 0$, and to $\ell_{p,q}- \ell_{p}$ for $0 < p < q$. We believe
that the usefulness of Theorem \ref{theorem} is well established, so we
refrain from giving further applications.
\end{remark}

Although our results concern spaceability of complements of linear subspaces, the same technique gives the spaceability of sets that are not related to linear subspaces at all. Rewriting the proof of Theorem \ref{theorem} we get:
\begin{proposition} Let $E$ be an invariant sequence space over the Banach space $X$. Let $A \subseteq E$ be such that:\\
{\rm (i)} For $x \in E$, $x \in A$ if and only if $x^0 \in A$.\\
{\rm (ii)} If $x = (x_j)_{j=1}^\infty \in A$ and $y = (y_j)_{j=1}^\infty \in E$ is such that $(\|y_j\|)_{j=1}^\infty $ is a multiple of a subsequence of $(\|x_j\|)_{j=1}^\infty $, then $y \in A$.\\
{\rm (iii)} There is $x \in E - A$ with $x^0 \neq 0$.\\
Then $E - A$ is spaceable.
\end{proposition}

\section{Norm-attaining operators}

In this section we show that the technique used in the previous section can be
used in a completely different context. Specifically, we extend a result from
\cite{eduardo} concerning the lineability of the set of norm-attaining operators.

Given Banach spaces $E$ and $F$ and $x_{0}\in E$ such that $\Vert x_{0}%
\Vert=1$ ($x_{0}$ is said to be a \textit{norm-one vector}), a continuous
linear operator $u\colon E\longrightarrow F$ attains its norm at $x_{0}$ if
$\Vert u(x_{0})\Vert=\Vert u\Vert$. By $\mathcal{N\!A}^{x_{0}}(E;F)$ we mean
the set of continuous linear operators from $E$ to $F$ that attain their norms
at $x_{0}$.

In \cite[Proposition 6]{eduardo} it is proved that if $F$ contains an
isometric copy of $\ell_{q}$ for some $1\leq q<\infty,$ then $\mathcal{N\!A}%
^{x_{0}}(E;F)$ is $\aleph_{0}$-lineable. We generalize this result showing
that this set is $c$-lineable:

\begin{proposition}
\label{Lineability of NA} Let $E$ and $F$ be Banach spaces so that $F$
contains an isometric copy of $\ell_{q}$ for some $1\leq q<\infty,$ and let
$x_{0}$ be a norm-one vector in $E.$ Then $\mathcal{N\!A}^{x_{0}}(E;F)$ is $c$-lineable.
\end{proposition}

\begin{proof}
The beginning of the proof follows the lines of the proof of \cite[Proposition 6]{eduardo}. It suffices to prove the result for $F=\ell_{q}$. Split $\mathbb{N}$
into countably many infinite pairwise disjoint subsets $(A_k)_{k=1}^\infty$.
For each positive integer $k$, write $A_{k}=\{a_{1}^{(k)}<a_{2}^{(k)}<\ldots\}$ and define%
\[
\ell_{q}^{(k)}:=\left\{  x\in\ell_{q}:x_{j}=0\text{ if }j\notin A_{k}\right\}
.
\]
Fix a non-zero operator $u \in \mathcal{N\!A}^{x_{0}}(E;F)$ and proceed as in the proof of \cite[Proposition 6]{eduardo} to get a sequence $(u^{(k)})_{k=1}^\infty$ of operators belonging to $\mathcal{N\!A}^{x_{0}%
}(E;\ell_{q}^{(k)})$ such that $\|u^{(k)}(x)\| = \|u(x)\|$ for every $k$ and every $x \in E$. By composing these operators with the inclusion
$\ell_{q}^{(k)} \hookrightarrow \ell_{q}$ we get operators (and we keep the notation $u^{(k)}$ for the sake of simplicity) belonging to $\mathcal{N\!A}^{x_{0}}(E;\ell_{q})$. For every $(a_{k})_{k=1}^{\infty} \in \ell_1$,
$$\sum_{k=1}^\infty \|a_ku^{(k)}\| = \sum_{k=1}^\infty |a_k|\|u^{(k)}\| = \sum_{k=1}^\infty |a_k| \|u^{(k)}(x_0)\| = \|u(x_0)\| \sum_{k=1}^\infty |a_k| < \infty, $$
so the map
$$T \colon \ell_{1}\longrightarrow\mathcal{L}(E;\ell_{q})~\,,\,~ T((a_{k})_{k=1}^{\infty})= {\displaystyle\sum\limits_{k=1}^{\infty}} a_{k}u^{(k)}$$
is well-defined. It is clear that $T$ is linear and injective. Hence $T(\ell_{1})$ is a $c$-dimensional subspace of $\ell_q$. Since the supports of the operators $u^{(k)}$ are pairwise disjoint, $T(\ell_{1})\subseteq\mathcal{N\!A}^{x_{0}}(E;\ell_{q})$.
\end{proof}

\medskip

\noindent[Geraldo Botelho and Vin\'icius V. F\'avaro] Faculdade de
Matem\'atica, Universidade Federal de Uberl\^andia, 38.400-902 - Uberl\^andia,
Brazil, e-mails: botelho@ufu.br, vvfavaro@gmail.com.

\medskip

\noindent[Diogo Diniz] UAME-UFCG, Caixa Postal 10044 - 58.109-970, Campina
Grande, Brazil, e-mail: diogodme@gmail.com.

\medskip

\noindent[Daniel Pellegrino] Departamento de Matem\'atica, Universidade
Federal da Pa\-ra\'iba, 58.051-900 - Jo\~ao Pessoa, Brazil, e-mail: dmpellegrino@gmail.com.

\end{document}